\documentclass[reqno]{amsart}
\usepackage{import}
\usepackage[T1]{fontenc}
\usepackage[utf8]{inputenc}
\usepackage{lmodern}

\usepackage[numbers]{natbib} 
\usepackage[english]{babel}

\usepackage{amsmath, amsfonts, amssymb, graphicx,amsthm, mathrsfs}
\usepackage{bbm}
\usepackage{amsfonts,amssymb,graphicx, tikz, mathabx, epsfig, listings}
\usepackage{tikz-cd}
\usepackage{hyperref}
\usepackage{caption}

\graphicspath{{pictures/}}

\newtheorem{thm}{Theorem}[section]
\newtheorem{prop}[thm]{Proposition}
\newtheorem{rem}[thm]{Remark}

\newtheorem{lem}[thm]{Lemma}

\newtheorem{cor}[thm]{Corollary}

\newtheorem{question}[thm]{Question}

\newtheorem*{prop*}{Proposition}
\newtheorem*{thm*}{Theorem}

\newtheorem*{rem*}{Remark}
\newtheorem*{imprem*}{Important Remark}
\newtheorem*{lem*}{Lemma}
\newtheorem*{dfn*}{Definition}
\newtheorem*{cor*}{Corollary}
\newtheorem*{probs*}{Problems}
\newtheorem*{prob*}{Problem}
\newtheorem*{problem*}{Problem}
\newtheorem*{ex*}{Example}
\newtheorem*{conj*}{Conjecture}
\newtheorem*{state*}{Statement}
\newtheorem*{question*}{Question}

\def\ep{\varepsilon}
\newcommand{\ie}[0]{\mathrm{i}}

\definecolor{vio}{RGB}{118, 120, 238}

\definecolor{identifiercolor}{rgb}{.4,.6,.56}
\definecolor{stringcolor}{gray}{0.5}
\definecolor{inactivecolor}{rgb}{0.15,0.15,0.5}

\title{A note on zero-density approaches for the difference between consecutive primes}
\author{Valeriia Starichkova}
\address{The University of New South Wales in Canberra}
\email{v.starichkova@unsw.edu.au}
\thanks{Supported by ARC Discovery Project DP240100186 and the Australian Mathematical Society Lift-off Fellowship of the author.}

\begin{document}

\maketitle

\section*{Abstract}

In this note, we generalise two results on prime numbers in short intervals. The first result is Ingham's theorem \cite{Ing37} which connects the zero-density estimates with short intervals where the prime number theorem holds,
and the second result is due to Heath-Brown and Iwaniec \cite{HBIw79}, which derives the weighted zero-density estimates used for obtaining the lower bound for the number of primes in short intervals. The generalised versions of these results make the connections between the zero-free regions, zero-density estimates, and the primes in short intervals more transparent. As an example, the generalisation of Ingham's theorem implies that, under the Density Hypothesis, the prime number theorem holds in $[x - \sqrt{x}\exp(\log^{2/3+\ep}x), x]$, which refines upon the classic interval $[x -x^{1/2+ \ep}, x]$.


\section{Introduction}

\subsection{Prime number theorem in short intervals}

Let $\Lambda(n)$ denote the von Mangoldt function, and $\psi(x) = \sum_{n \leq x} \Lambda(x)$ denote the corresponding Chebyshev function. We say that the \emph{prime number theorem holds} in the intervals $(x - y, x]$, $y = y(x)$, if 
\begin{equation} \label{eq: pnt-prop}
    \psi(x) - \psi(x-y) \sim y \quad \text{as} \quad x \to \infty.
\end{equation}

One way to derive the length $y$ for which \eqref{eq: pnt-prop} holds is via the error term $E(x) = \psi(x) - x$ in the Prime Number Theorem (PNT). The size of $E(x)$ is tightly connected with the zero-free regions for the Riemann zeta function $\zeta$. In \cite{Ingham1932}, Ingham showed that if $\zeta(s)$, $s = \sigma + it$, does not have any zeroes in the region $\sigma > 1 - \eta(T)$ and $\eta(T)$ satisfies some extra properties, then
\begin{equation} \label{eq: PNT-error-Ingham}
    E(x) \ll_{\ep} x \exp\left( - \frac{1}{2}(1 - \ep) ~\omega_{(\eta)}(x) \right),
\end{equation}
where
\begin{equation*}
    \omega_{(\eta)}(x) = \inf_{T\geq 1} (\eta(T) \log x + \log T).
\end{equation*}

In \cite[Theorem 8]{Pintz1984}, Pintz improved $\frac{1}{2}(1 - \ep)$ to $(1 - \ep)$ on the right-hand side of \eqref{eq: PNT-error-Ingham}. Thus, currently the best asymptotic zero-free regions due to Korobov \cite{Korobov1958} and Vinogradov \cite{Vinogradov1958},
\begin{equation} \label{eq: Korobov-Vinogradov}
    \beta \geq 1 - \frac{c}{(\log t)^{2/3} (\log \log t)^{1/3}}, \quad t \gg 1,
\end{equation}
where $c > 0$ is a constant, imply
\begin{equation*}
    E(x) \ll_{\ep} x \exp\left( - d_{\ep} (\log x)^{3/5} (\log \log x)^{-1/5} \right),
\end{equation*}
where $d_{\ep} = \left(\frac{5^6 c^3}{2^2 \cdot 3^4}\right)^{1/5} - \ep$, see \cite[Section 6]{JohnstonYang2023}.

The bound \eqref{eq: PNT-error-Ingham} and its improvement by Pintz imply that the prime number theorem holds in $(x - y, x]$ as long as for some $\ep > 0$, $x \exp\left( - (1 - \ep) ~\omega_{(\eta)}(x) \right) = o\left( y \right)$. This establishes a connection between zero-free regions and intervals where the prime number theorem holds. 

Hoheisel \cite{Hoheisel37} showed the prime number theorem holds in shorter intervals by using instead the zero-density estimates for $\zeta$: he showed that \eqref{eq: pnt-prop} holds for $y = x^{\theta}$ for any $0 \leq \theta < 1 - \frac{1}{33000}$. The approach of Hoheisel, which we will refer to as the \emph{classic zero-density approach}, is well-described, simplified, and improved by Ingham \cite{Ing37}. Namely \cite[Theorem 1]{Ing37} states that
\begin{itemize}
    \item the zero-free regions for $\zeta(\sigma + i T)$ of the form $\sigma > 1 - \eta(T)$ with $\frac{\log \log T}{\log T} = o(\eta(T))$, and
    \item the zero-density estimates of the form $N(\sigma, T) \ll T^{b(1 - \sigma)} \log^B T$, which hold for constants $b, B > 0$ uniformly in $\frac{1}{2} \leq \sigma \leq 1$,
\end{itemize}
imply that the prime number theorem holds in $(x - y, x]$, $y = x^{\theta}$, if $1 - \frac{1}{b} < \theta < 1$. In particular, the recently obtained in \cite{GuthMaynard2024} value $b = \frac{30}{13} + \ep$ implies $\frac{17}{30} < \theta < 1$.

Thus, the length $y$ is mainly influenced by the zero-density estimates. In contrast, the only property required from the zero-free regions is to be ``good enough''. In this way, the connection between zero-free regions and $y$, explicit in \eqref{eq: PNT-error-Ingham}, disappears. However, zero-free regions do affect $y$ up to $x^{o(1)}$. This can be compared with the influence of zero-free regions on the error term $E(x) = x^{1 - o(1)}$ in the PNT. This is why the contribution of the zero-free regions is not visible if we focus on the case $y = x^{\theta}$ where $\theta$ is a constant. In Theorem \ref{thm: Ing37-1-gen}, we generalise Ingham's result to intervals of length $y = x^{\theta} g(x)$, $g(x) = x^{o(1)}$, to show explicitly the dependence of $y$ on the combination of zero-free regions and zero-density estimates. This allows us to reach slightly stronger conclusions about the length $y$. For example, we show that under the Density Hypothesis, \eqref{eq: pnt-prop} holds for $y = \sqrt{x} \exp(\log^{2/3 + \ep}x)$, which improves upon the classic length $y = x^{1/2 + \ep}$.

\subsection{The existence of primes in short intervals}
The relation \eqref{eq: pnt-prop} is a particular case of the following weighted relation
\begin{equation} \label{eq: PNT-weighted}
    \sum_{\substack{m, n \\ x-y < mnr \leq x}} a_m b_n \Lambda(r) \sim y \sum_{m, n} \frac{a_m b_n}{mn},
\end{equation}
where integers $m$ and $n$ are restricted to some (usually, dyadic) regions, and coefficients $a_m, b_n$ are real numbers between $0$ and $1$. The estimates of the form \eqref{eq: PNT-weighted} were used in the works \cite{HBIw79} and \cite{BakHar96} to bound the number of primes in short intervals from below. The obtained bounds were of the form
\begin{equation} \label{eq: primes-lower}
    \pi(x) - \pi(x-y) \gg c \frac{y}{\log x} \text{ as } x \to \infty,
\end{equation}
where $c > 0$ is a constant. In \cite{HBIw79}, Heath-Brown and Iwaniec showed \eqref{eq: primes-lower} for $y = x^{\theta}$, $\frac{13}{23} \leq \theta < 1$, and in \cite{BakHar96}, Baker and Harman pushed the value of $\theta$ down to $0.535$. Both results improve upon $\theta = \frac{17}{30} + \ep$ from the previous section obtained using the classic zero-density approach.

To obtain \eqref{eq: PNT-weighted}, the works \cite{HBIw79} and \cite{BakHar96} used the \emph{weighted zero-density estimates}, i.e.\ the estimates of the form \eqref{eq: new-dens-weighted-gen} and \eqref{eq: BH-zero-dens} below. We refer the reader to Section \ref{sec: appendix} to see how to derive \eqref{eq: PNT-weighted} from such estimates. The weighted estimates were derived from the \emph{classic} zero-density estimates of the form $N(\sigma, T) \ll T^{A(\sigma)(1-\sigma)} \log^B T$. However, the mentioned papers work with particular choices of $A(\sigma)$, so it is not straightforward to see which properties of $A(\sigma)$, in general, are crucial for improving upon the weighted zero-density estimates. In section \ref{sec: New-Zero-Density}, we will list these properties in a simple form. This simplification allows us to improve upon the weighted zero-density estimates in \cite{BakHar96}, see Corollary \ref{cor: new-dens-weighted}.

Namely, we focus on the weighted zero-density estimates \cite[Lemma 6]{HBIw79} and \cite[Lemma 1(i)]{BakHar96}, which provided major contributions in the corresponding papers. We generalise these two lemmas in Lemmas \ref{lem: new-dens-weighted-gen} and \ref{lem*: new-dens-weighted} respectively. We note that \cite[Lemma 6]{HBIw79} and \cite[Lemma 1(i)]{BakHar96} (and thus our generalisations) have very similar statements, but slightly different forms of the weighted zero-densities. We explain how to switch from one form to another in Proposition \ref{prop: zero_dens_transition}.

Finally, we weaken an initial condition from \cite[Lemma 1(i)]{BakHar96} to allow more flexibility. Namely, whereas \cite[Lemma 6]{HBIw79} has to rely on the explicit version of the zero-density estimates $N(\sigma, T) \ll T^{A(\sigma)(1-\sigma)} \log^B T$ (meaning that the $T^{o(1)}$ term is explicitly written as $\log^B T$), this is not the case any more in Lemma \cite[Lemma 1(i)]{BakHar96}: we prove its generalisation, Lemma \ref{lem*: new-dens-weighted}, assuming the non-explicit bound $N(\sigma, T) \ll T^{A(\sigma)(1-\sigma) + o(1)}$. This is useful for applications since many zero-density estimates in the literature are proven in the non-explicit form.

We note that although weighted zero-density estimates were not used in recent works on primes in short intervals (\cite{BHP}, \cite{MatMerTer2024}), this does not necessarily mean that the result from \cite{BHP} is out of reach for the weighted zero-density approach. However, the computational complexity and the lack of clarity in how to use the available zero-density estimates to maximise the regions for $(m,n)$ in \eqref{eq: PNT-weighted}, make this approach harder to use. The simpler connections one derives, the easier it will be to apply new improvements on zero-density results, and one of the aims of this paper is to make such connections simpler and more transparent.

\section{The classic zero-density approach}

\subsection{Ingham's theorem and its generalisation}

In \cite{Ing37}, Ingham provides two very valuable results: Theorem 3, which connects the bounds for $\zeta$ on the critical line with the zero-density estimates, and Theorem 1 connecting the zero-density estimates and zero-free regions with primes in short intervals. We will focus on the latter result which we cite below. 
We denote by $N(\sigma, T)$ the number of zeroes $\rho = \beta + \ie \gamma$ of the Riemann zeta function in the rectangle $\sigma \leq \beta \leq 1$, $0 < \gamma \leq T$.

\begin{thm}{\cite[Theorem 1]{Ing37}} \label{thm: Ing37-1}
Suppose $A > 0, B \geq 0$, and $\zeta(s)$, $s = \sigma + \ie t$, has no zeroes in the domain
\begin{equation*}
    \sigma > 1 - A \frac{\log \log t}{\log t} \quad \text{for } t > t_0 > 3,
\end{equation*}
and that
\begin{equation*}
    N(\sigma, T) = O\left( T^{b(1 - \sigma)} \log^B T \right),
\end{equation*}
uniformly for $\frac{1}{2} \leq \sigma \leq 1$ as $T \to \infty$. Then,
\begin{equation} \label{eq: Ingham-primes-asymp}
    \pi(x) - \pi(x - x^{\theta}) \sim \frac{x^{\theta}}{\log x} \quad \text{as } x \to \infty,
\end{equation}
for any $\theta$ satisfying
\begin{equation} \label{eq: Ing37-theta-bound}
    1 - \frac{1}{b + A^{-1}B} < \theta < 1.
\end{equation}
\end{thm}

Let us note that according to the modern results on zero-free regions, $A \to \infty$ as $t_0 \to \infty$, so \eqref{eq: Ing37-theta-bound} can be reduced to $1 - \frac{1}{b} < \theta < 1$. Hence, the smaller the value of $b$ is, the better the bound for $\theta$ we get. Ingham got $b = \frac{925}{348} + \ep \approx 2.6581 + \ep$, and thus $\theta = \frac{577}{925} + \varepsilon$, using the upper bounds for $\zeta(s)$ on the critical line $\Re(s) = \frac{1}{2}$. Here, $\ep$ conventionally denotes a small positive number and does not need to be the same at each appearance. Meanwhile, a lower value of $b$ can be obtained with different techniques: in 1972, Huxley \cite{Hux72} got $b = \frac{12}{5} = 2.4$, and very recently Guth and Maynard \cite{GuthMaynard2024} improved Huxley's result to $b = \frac{30}{13} + \ep \approx 2.308 + \ep$. The last result implies \eqref{eq: Ingham-primes-asymp} for all $\theta > \frac{17}{30}$.

Under the Density Hypothesis (DH), i.e.\ if $b = 2 + \varepsilon$, Theorem \ref{thm: Ing37-1} implies $\theta = \frac{1}{2} + \varepsilon$, which is slightly worse than what the Riemann Hypothesis (RH) implies.
\begin{thm}[{\cite[Theorem I]{Cramer1920}}]
Assume RH holds, then there exists a constant $c > 0$ such that
\begin{equation} \label{eq: Cramer}
    \pi(x + c \sqrt{x} \log x) - \pi(x) > \sqrt{x} 
\end{equation}
for all $x \geq 2$.
\end{thm}

Our generalisation of Theorem \ref{thm: Ing37-1} below distinguishes between RH and DH.

\begin{thm} \label{thm: Ing37-1-gen}
Let $\frac{1}{2} \leq \theta < 1$ and $y = x^{\theta} g(x) \leq \frac{x}{2}$, where $\log^{\delta} x \ll g(x)$ for some $\delta > 0$, and $g(x) \ll_\ep x^{\varepsilon}$ for all $\varepsilon > 0$. Then
\begin{equation*}
    \pi(x) - \pi(x-y) \sim \frac{y}{\log x},
\end{equation*}
if the conditions below hold.
\begin{enumerate}
    \item $\zeta(s)$, $s = \sigma + \ie t$, has no zeroes in the region
    \begin{equation*}
        \sigma > 1 - \eta(T) \quad \text{for } T \geq T_0 \geq 3,
    \end{equation*}
    where $\eta(T) < \frac{1}{2}$ is a non-increasing and non-negative function.
    \item There exist a non-decreasing function $h(T)$ and a constant $0 < b \leq \frac{1}{1-\theta}$ such that
    \begin{equation} \label{eq: Ing37-1-gen-zero-density}
        N(\sigma, T) \ll T^{b(1-\sigma)} h(T),
    \end{equation}
    uniformly for $\frac{1}{2} \leq \sigma \leq 1$ as $T \to \infty$.
    \item The functions $g, \eta,$ and $h$ satisfy the following relation as $x$ tends to $\infty$.
    \begin{itemize}
    \item If $\theta > 1 - \frac{1}{b}$, then
    \begin{equation} \label{eq: Ing-gen-asmpt-3-1}
        h(x) = o \left( x^{(1 + b(\theta-1)) \eta(x)} \right).
    \end{equation}
    \item If $\theta = 1 - \frac{1}{b}$, then
    \begin{equation} \label{eq: Ing-gen-asmpt-3-2}
        \left( \frac{\log^{\delta/2}x}{g(x)} \right)^{b \eta(x)} \frac{h(x) \log x}{\log \log x} = o(1).
    \end{equation}
    \item In addition, if $b = 2$ and $\theta = \frac{1}{2}$, then
    \begin{equation} \label{eq: Ing-gen-asmpt-3-3}
        \log^{1 + \delta} x \ll g(x).
    \end{equation}
    \end{itemize}
\end{enumerate}
\end{thm}
\begin{proof}
We will follow the proof of \cite[Theorem 1]{Ing37} and modify some of its parts. Let us consider $\psi(x) = \sum_{n \leq x} \Lambda(n)$, then Perron's formula \cite[Lemma 3.12]{Tit86} would give us
\begin{equation*}
    \psi(x) - \psi(x-y) = y - \sum_{|\gamma| \leq T} \frac{x^{\rho} - (x-y)^{\rho}}{\rho} + O\left( \frac{x}{T} \log^2 x \right),
\end{equation*}
for some $1 \leq T \leq x$. In fact, the error term was refined to $O\left(\frac{x}{T} \log x \log \log x\right)$ by Goldston in \cite{Goldston1983}, and improved on average by Wolke in \cite{Wolke1983}. If we use Wolke's bound for $4 \leq T \ll x^{1-\ep}$ with $0 < \ep < 1$, we will get
\begin{equation} \label{eq: truncated_Perron_short_sharper}
    \psi(x) - \psi(x-y) = y - \sum_{|\gamma| \leq T} \frac{x^{\rho} - (x-y)^{\rho}}{\rho} + O_{\ep}\left( \frac{x}{T} \right).\\
\end{equation}

We will show that the right-hand side of \eqref{eq: truncated_Perron_short_sharper} is $y + o(y)$ to prove the theorem. We choose $T = \frac{x}{y} (\log{x})^{\delta/2}$ with $\delta$ from the definition of $g(x)$. This choice bounds the third term on the right-hand side of \eqref{eq: truncated_Perron_short_sharper} by $O\left( y\log^{-\delta/2} x \right)$. To get a bound for the second term we use
\begin{align}
    \left| \frac{x^{\rho} - (x - y)^{\rho}}{\rho} \right| &= \left| \int_{x - y}^{x} u^{\rho - 1} du \right| \leq \int_{x - y}^{x} u^{\beta - 1} du \nonumber\\
    &\leq y (x-y)^{\beta-1} \leq y x^{\beta - 1} \left( 1 - \frac{y}{x} \right)^{\beta - 1} \leq 2 y x^{\beta - 1}, \label{eq: mean-value-x-rho}
\end{align}
for all $\beta > 0$.
It remains to prove that
$$\sum_{\substack{\rho \\ |\gamma| \leq T}} x^{\beta-1} = o(1).$$
For a zero $\rho$ of $\zeta$ we write $\rho = \beta + \ie \gamma$, then
\begin{equation*}
    \sum_{\substack{\rho \\ |\gamma| \leq T}} (x^{\beta-1} - x^{-1}) = \sum_{\substack{\rho \\ |\gamma| \leq T}} \int_{0}^{\beta} x^{\sigma - 1} \log x ~d\sigma = \int_{0}^{1} \sum_{\substack{\rho \\ |\gamma|\leq T\\ \beta \geq \sigma}} x^{\sigma - 1} \log x ~d \sigma,
\end{equation*}
whence
\begin{align}
    \sum_{\substack{\rho \\ |\gamma| \leq T}} x^{\beta-1} &= 2 x^{-1} N(0, T) + 2 \int_{0}^{1} N(\sigma, T) x^{\sigma - 1} \log x ~d\sigma \nonumber\\
    &= 2 x^{-1} N(0, T) + 2\int_{0}^{1-\eta(T)} N(\sigma, T) x^{\sigma - 1} \log x ~d\sigma. \label{eq: Ing37-dens-integral}
\end{align}
Since
\begin{equation} \label{eq: N-sigma-omega}
    N\left(\frac{1}{2}, T\right) = \Omega(T \log T),
\end{equation}
by \cite[Theorem 10.7]{Tit86}, we get $b \geq 2$ by substituting $\sigma = \frac{1}{2}$ into \eqref{eq: Ing37-1-gen-zero-density}. Moreover,  by \cite[Theorem 9.4]{Tit86}, $$N(0, T) \sim \frac{T}{2\pi} \log T,$$ hence we can replace \eqref{eq: Ing37-1-gen-zero-density} by the stronger bound $N(\sigma, T) \ll T \log T$ for all $0 \leq \sigma \leq \frac{1}{2}$. Thus, the first term in \eqref{eq: Ing37-dens-integral} is $O(x^{-1} T \log T)$, and we can split the range for the integral in \eqref{eq: Ing37-dens-integral} into two parts $0 \leq \sigma \leq \frac{1}{2}$ and $\frac{1}{2} \leq \sigma \leq 1 - \eta(T)$. Hence, \eqref{eq: Ing37-dens-integral} is asymptotically bounded from above by
\begin{align}
    x^{-1} T \log T + \int_{0}^{1-\eta(T)} &N(\sigma, T) x^{\sigma - 1} \log x ~d\sigma \ll x^{-1/2} T \log T + \int_{1/2}^{1-\eta(T)} N(\sigma, T) x^{\sigma - 1} \log x ~d\sigma \nonumber\\
    &\ll \frac{\log^{1 + \delta/2} x}{g(x)} x^{1/2 - \theta} + \left[ \left( \frac{x}{T^b} \right)^{-\eta(T)} - \left( \frac{x}{T^b} \right)^{-1/2} \right] \frac{h(T) \log x}{\log \frac{x}{T^b}}. \label{eq: N-sigma-T-integral}
\end{align}

Let us show that the first term from the second line in \eqref{eq: N-sigma-T-integral} is $o(1)$. If $\theta > \frac{1}{2}$, it follows from $g(x) \gg 1$. If $\theta = \frac{1}{2}$, it follows from assumption \eqref{eq: Ing-gen-asmpt-3-3}. By definition, $T^b < x^{b(1 - \theta)}$ for $x$ large enough, hence $T^b < x$ because $\theta \geq 1 - \frac{1}{b}$. Thus, $\frac{x}{T^b} > 1$ and the expression in the square brackets in \eqref{eq: N-sigma-T-integral} is asymptotically $\left( \frac{x}{T^b} \right)^{-\eta(T)}$ because $\eta(T) < \frac{1}{2}$ by definition. We finally have
\begin{equation*}
    \sum_{\substack{\rho \\ |\gamma| \leq T}} x^{\beta-1} \ll \left( \frac{\log^{\delta/2}x}{g(x)} ~x^{1-\theta - 1/b} \right)^{b \eta(T)} \cdot \frac{h(T) \log x}{b \log \left( \frac{g(x)}{\log^{\delta/2}x} ~x^{\theta + 1/b - 1} \right)} + o(1).
\end{equation*}
Let us assume that $\theta + \frac{1}{b} - 1 > 0$. Since $g(x) \gg \log^{\delta} x$, we have
\begin{equation*}
    \left( \frac{\log^{\delta/2}x}{g(x)} ~x^{1-\theta - 1/b} \right)^{b \eta(T)} \ll x^{(1 - \theta - 1/b) b\eta(T)}.
\end{equation*}
Thus, we get
\begin{equation*}
    \sum_{\substack{\rho \\ |\gamma| \leq T}} x^{\beta-1} \ll x^{(1 - \theta - 1/b) b\eta(T)} h(T) = o(1),
\end{equation*}
by assumption \eqref{eq: Ing-gen-asmpt-3-1} and the monotonicity of $\eta(x)$ and $h(x)$.

In case $\theta + \frac{1}{b} - 1 = 0$, we use $\log^{\delta/2} \ll g(x)$ and the monotonicity of $\eta(x)$ and $h(x)$ to get:
\begin{align*}
    \sum_{\substack{\rho \\ |\gamma| \leq T}} x^{\beta-1} \ll_\delta \left( \frac{\log^{\delta/2}x}{g(x)} \right)^{b \eta(T)} \frac{h(T) \log x}{\log \log x} \ll \left( \frac{\log^{\delta/2}x}{g(x)} \right)^{b \eta(x)} \frac{h(x) \log x}{\log \log x} = o(1),
\end{align*}
where the last assertion follows from assumption \eqref{eq: Ing-gen-asmpt-3-2}. This concludes the proof of the theorem.
\end{proof}

As an easy corollary, we can bound $g(x)$ by a finite logarithm power under a very strong assumption that the zero-free regions contain a vertical strip on the left of $\Re(s) = 1$.

\begin{cor} \label{cor: under-DH}
Let us assume the classical zero-density estimates hold
$$N(\sigma, T) \ll T^{b(1-\sigma)} \log^B T,$$
uniformly for $\frac{1}{2} \leq \sigma \leq 1$, where $b$ and $B$ are positive constants. Let us assume there exists $0 < \eta_0 < \frac{1}{2}$ such that $\zeta(s)$ has no zeroes for $\Re(s) > 1 - \eta_0$. Then
$$\pi(x) - \pi(x-y) \sim \frac{y}{\log x},$$
for $y = x^{1 - \frac{1}{b}} \log^C x$ for all $C$ satisfying
$$C > \max \left\{ 1, \frac{B+1}{b\eta_0} \right\}.$$

In particular under RH, for any $\ep > 0$, the prime number theorem holds in $(x-y, x]$ with $y = \sqrt{x}\log^{2 + \ep} x$.
\end{cor}

\begin{rem}
\textnormal{Under RH, \eqref{eq: Cramer} gives a stronger result than Corollary \ref{cor: under-DH} since the approach from Theorem \ref{thm: Ing37-1-gen} is not sensitive to the information about all zeroes of $\zeta$ lying on one vertical line.}
\end{rem}

In addition, we obtain the following statement under DH.
\begin{cor}
Let us assume DH holds, that is
$$N(\sigma, T) \ll T^{2(1-\sigma)} \log^B T,$$
uniformly for $\frac{1}{2} \leq \sigma \leq 1$ and a constant\footnote{The condition $B < 1$ would contradict the bound $N\left(\frac{1}{2}, T\right) = \Omega(T \log T)$ \eqref{eq: N-sigma-omega}.} $B \geq 1$. Then
$$\pi(x) - \pi(x-y) \sim \frac{y}{\log x},$$
for $y = \sqrt{x} \exp(\log^{\alpha} x)$ for all $2/3 < \alpha < 1$.
\end{cor}
The corollary above is slightly stronger than $y = x^{1/2 + \ep}$, which is usually stated as an implication of DH, see \cite[Chapter 14]{Mont71} for example.
\begin{proof}
    We apply Theorem \ref{thm: Ing37-1-gen} with the following parameters:
    \begin{align*}
        b = 2, \quad \theta = \frac{1}{2}, \quad h(x) = \log^B x,\\
        \eta(x) = \frac{1}{48.08( \log x)^{2/3} (\log \log x)^{1/3}},
    \end{align*}
    where $\eta(x)$ is chosen as in \eqref{eq: Korobov-Vinogradov} with $c = (48.08)^{-1}$ \cite{Belotti2024}. Then \eqref{eq: Ing-gen-asmpt-3-2} from Theorem \ref{thm: Ing37-1-gen} follows from
    $$\left( \log^{B+1} x \right)^{\frac{1}{2\eta(x)}} = o\left( \frac{g(x)}{\log^{\delta/2} x} \right),$$
    which is equivalent to
    $$\left( \log x \right)^{\frac{48.08(B+1)}{2} ( \log x)^{2/3} (\log \log x)^{1/3} + \frac{\delta}{2}} = o(g(x)),$$
    which is true if
    $$\frac{48.08(B+1)}{2} ( \log x)^{2/3} (\log \log x)^{1/3} + \frac{\delta}{2} = o(\log g(x)).$$
    Thus, choosing $\alpha > \frac{2}{3}$ and $g(x) = \exp(\log^{\alpha} x)$ is sufficient for satisfying the assumptions of Theorem \ref{thm: Ing37-1-gen}.
\end{proof}


\section{Weighted zero-density estimates} \label{sec: New-Zero-Density}
In the previous section, we discussed how the bounds of the form 
\begin{equation} \label{eq: classic-zero-dens}
    N(\sigma, T) \ll T^{b(1 - \sigma) + o(1)},
\end{equation}
with $b \geq 2$ a non-negative constant, imply the asymptotic for primes in short intervals $(x - x^{\theta}, x]$ for all $1 - \frac{1}{b} < \theta < 1$ and $x \gg_{\theta, b} 1$ large enough. In this section, we will describe the approach from \cite{HBIw79} by Heath-Brown and Iwaniec used to switch from the bounds of the form \eqref{eq: classic-zero-dens}, which we will refer to as \emph{classic zero-density estimates}, to the \emph{weighted zero-density estimates}, as in \eqref{eq: new-dens-weighted-gen} below. The weighted estimates were used by Heath-Brown and Iwaniec, and in a slightly different form by Baker and Harman \cite{BakHar96}, to show the existence of primes in short intervals. We will cover both forms in Sections \ref{subsec: New-Zero-Density: weighted-HB-Iw} and \ref{subsec: New-Zero-Density: weighted-B-H} respectively. 

As a first step toward weighted zero-density estimates, we should note that the bounds of the form \eqref{eq: classic-zero-dens} are usually proved by deriving the stronger bounds of the form
$$N(\sigma, T) \ll T^{A(\sigma)(1-\sigma) + o(1)},$$
where the maximum of $A(\sigma)$ for $0 \leq \sigma \leq 1$ is equal $b$. Though this observation is not relevant when one uses the methods from \cite{Ing37}, it becomes useful for deriving better weighted zero-density estimates, thus implying better results for primes in short intervals.

\subsection{On weighted zero-density estimates by Heath-Brown and Iwaniec}
\label{subsec: New-Zero-Density: weighted-HB-Iw}

In this section, we present a generalisation for the weighted zero-density estimates \cite[Lemma 6]{HBIw79}.
We will use the following notations:

$$m \sim M \text{ means } M \leq m < 2M,$$
$$\{a_m\}_{m \in \mathbb{N}}, \{b_n\}_{n \in \mathbb{N}}, \quad 0 \le a_m \le 1, \quad 0 \le b_n \le 1,$$
$$M(s) = \sum_{m \sim M} a_m m^{-s}, \quad N(s) = \sum_{n \sim N} b_n n^{-s}.$$

\begin{lem} \label{lem: new-dens-weighted-gen}
Let $\frac{1}{2} \leq \theta < 1$, $\frac{1}{2} < \sigma_0 < 1$, and let $A(\sigma)$, $0 \leq \sigma \leq 1$, be a non-negative function satisfying the following properties:
\begin{enumerate}
    \item $A(\sigma_0) = \frac{1}{2(1-\sigma_0)}$;
    \item $A(\sigma)$ is monotonously non-increasing for $\sigma \geq \sigma_0$.
\end{enumerate}
Let us assume that there exists a constant $B$ such that
\begin{equation} \label{eq: zero-dens-gen}
    N(\sigma, T) \ll T^{A(\sigma)(1-\sigma)} \log^B T,
\end{equation}
uniformly for $\frac{1}{2} \leq \sigma \leq 1$. Denote $T = x^{1 - \theta}$, $M = x^u$, and $N = x^v$ such that the following conditions hold
\begin{align*}
    \frac{2 \sigma_0 - 2 \theta}{2 \sigma_0 - 1} &\leq u + v \leq 1, &(A1) \\ 
    u (2 - 2 \sigma_0) + v (1 - 2 \sigma_0) &\leq \theta + 1 - 2 \sigma_0, &(A2) \\
    v (2 - 2 \sigma_0) + u (1 - 2 \sigma_0) &\leq \theta + 1 - 2 \sigma_0, &(A3)\\
    \theta &\leq \frac{3 - 2\sigma_0}{7 - 6\sigma_0}. &(A4)
\end{align*}
Let $\{a_m\}$, $\{b_n\}$, $M(s)$, and $N(s)$ be defined as above. Then
\begin{equation} \label{eq: new-dens-weighted-gen}
    \sum_{\substack{\rho \\ \beta \geq \sigma, |\gamma| \leq T}} \left| M(\rho) N(\rho) \right| \ll x^{1-\sigma} (\log x)^c,
\end{equation}
for a constant $c$ depending on $B$, uniformly for $0 \leq \sigma \leq 1$. The summation is over all non-trivial zeroes $\rho = \beta + \ie \gamma$ of $\zeta$.
\end{lem}

\begin{proof}
We will start by following the proof of \cite[Lemma 5]{HBIw79}. By the Cauchy--Schwarz inequality,
\begin{equation} \label{eq: Holder-2}
    \sum |M(\rho) N(\rho)| \leq \left( \sum \left| M(\rho) \right|^2 \right)^\frac{1}{2} \left( \sum \left| N(\rho) \right|^2 \right)^\frac{1}{2},
\end{equation}
where the right-hand side can be bounded by the estimates for the mean values of the Dirichlet polynomials \cite[Theorem 7.3]{Mont71} and the Hal\'asz inequality \cite[Theorem 8.2]{Mont71}
\begin{equation} \label{MeanDirPol}
    \sum_{\substack{\rho \\ \beta \geq \sigma, |\gamma| \leq T}} \left| M(\rho) \right|^2 \ll M^{1 - 2\sigma} \left( M + \min(T, T^{\frac{1}{2}} N(\sigma, T)) \right) \log T.
\end{equation}
We can bound $N(\sigma, T)$ by \eqref{eq: zero-dens-gen}, hence by denoting $E = \min\left( T, T^{\frac{1}{2} + A(\sigma)(1 - \sigma)} \right)$ we replace \eqref{MeanDirPol} by the following stronger condition:
$$\sum_{\substack{\rho \\ \beta \geq \sigma, |\gamma| \leq T}} \left| M(\rho) \right|^2 \ll M^{1 - 2\sigma} \left( M + E \right) \log^{B+1} T,$$
and the same estimates are valid for $\sum_{\rho} \left| N(\rho) \right|^2$.
The statement of the lemma thus follows from
\begin{equation} \label{EstPowersAndLogs}
M^{\frac{1}{2} - \sigma} N^{\frac{1}{2} - \sigma} (M + E)^{\frac{1}{2}} (N + E)^{\frac{1}{2}} \log^{B+1} x \ll x^{1 - \sigma} \log^c x.
\end{equation}
We choose $c \geq B + 1$, so that the upper bound \eqref{EstPowersAndLogs} is implied by
\begin{equation} \label{EstPowers}
    M^{\frac{1}{2} - \sigma} N^{\frac{1}{2} - \sigma} (M + E)^{\frac{1}{2}} (N + E)^{\frac{1}{2}} \ll x^{1 - \sigma}.
\end{equation}
After dividing each side of \eqref{EstPowers} by $(MN)^{1-\sigma}$ we obtain
\begin{equation*}
(1 + E/M)^{\frac{1}{2}} (1 + E/N)^{\frac{1}{2}} \ll \left( \frac{x}{MN}\right)^{1 - \sigma},
\end{equation*}
which follows from the inequality on exponents with base $x$
\begin{equation} \label{EstLogs}
\frac{1}{2}\max\left( 0, \log_x E - u \right) + \frac{1}{2}\max\left( 0, \log_x E - v \right) \leq ( 1 - u - v) (1 - \sigma).
\end{equation}
We consider two cases.\\

\noindent
{\bfseries Case 1.} $0 \leq \sigma \leq \sigma_0$\\
Then $E \leq T = x^{1 - \theta}$, and \eqref{EstLogs} will follow from the inequality
\begin{equation} \label{eq: EstLogs-case1}
\frac{1}{2}\max\left( 0, 1 - \theta - u \right) + \frac{1}{2}\max\left( 0, 1 - \theta - v \right) \leq ( 1 - u - v) (1 - \sigma). 
\end{equation}
If $u, v \geq 1 - \theta$, then \eqref{eq: EstLogs-case1} is equivalent to
$$0 \leq (1 - u - v) (1 - \sigma),$$
which follows from $\sigma \leq 1$ and the upper bound for $u + v$ from $(A1)$. If $u \leq 1 - \theta \leq v$, then \eqref{eq: EstLogs-case1} is
$$\frac{1}{2}(1 - \theta - u) \leq (1 - u - v)(1 - \sigma).$$
The right-hand side is non-increasing in $\sigma$, so we just need to show that the last inequality is true for $\sigma = \sigma_0$. The inequality
$$\frac{1}{2}(1 - \theta - u) \leq (1 - u - v) (1 - \sigma_0)$$
is equivalent to the initial condition $(A3)$.
The argument for $v \leq 1 - \theta \leq u$ is similar with the roles of $u$ and $v$ reversed. If $u, v \leq 1 - \theta$, then \eqref{eq: EstLogs-case1} is the same as
$$1 - \theta - \frac{1}{2}(u + v) \leq (1 - u - v)(1 - \sigma),$$
and follows from $\sigma \leq \sigma_0$ and the lower bound on $u + v$ from $(A1)$.\\

\noindent
{\bfseries Case 2.} $\sigma_0 \leq \sigma \leq 1$\\
The properties of the function $A(\sigma)$ imply that $A(\sigma)(1 - \sigma)$ is non-increasing for $\sigma \geq \sigma_0$. Hence for all $\sigma \geq \sigma_0$, $A(\sigma)(1 - \sigma) \leq A(\sigma_0)(1 - \sigma_0) = \frac{1}{2}$. Thus, $E = T^{\frac{1}{2} + A(\sigma)(1 - \sigma)}$ and the inequality \eqref{EstLogs} becomes
\begin{align}
&\frac{1}{2}\max\left( 0, (1 - \theta) \left( \frac{1}{2} + A(\sigma)(1 - \sigma) \right) - u \right) \nonumber\\
+ &\frac{1}{2}\max\left( 0, (1 - \theta) \left( \frac{1}{2} + A(\sigma)(1 - \sigma) \right) - v \right) \leq ( 1 - u - v) (1 - \sigma). \label{eq: EstLogs-case2}
\end{align}

In case $u, v \geq (1 - \theta) \left( \frac{1}{2} + A(\sigma)(1 - \sigma) \right) - u$, the inequality \eqref{eq: EstLogs-case2} follows from $\sigma \leq 1$ and the upper bound for $u + v$ from $(A1)$.

Before proceeding to the next cases, we note that in the course of the PNT proof, Hadamard and de la Vall\'ee Poussin \cite[Chapter 3]{Tit86} independently proved that $\zeta(s)$ does not vanish on the line $\Re(s) = 1$. Hence, if $\sigma = 1$, then the sum on the left-hand side of \eqref{eq: new-dens-weighted-gen} is empty thus making the inequality hold. Thus, from now on we can assume that $\sigma < 1$.

Let $u \leq (1 - \theta) \left( \frac{1}{2} + A(\sigma)(1 - \sigma) \right) \leq v$, then \eqref{eq: EstLogs-case2} becomes
$$\frac{1}{2} (1 - \theta) \left( \frac{1}{2} + A(\sigma)(1 - \sigma) \right) - \frac{1}{2} u \leq ( 1 - u - v) (1 - \sigma),$$
and can be rearranged as follows
\begin{equation}
    \frac{\frac{1-\theta}{2} - u}{1 - \sigma} + A(\sigma) (1 - \theta) \leq 2 (1 - u - v). \label{eq: case2-u<v}
\end{equation}
From the conditions $(A1)$ and $(A3)$ we have
$$u \geq \frac{2\theta \sigma_0 - 3\theta + 1}{2\sigma_0 - 1},$$
which, combined with $(A4)$, implies $u \geq \frac{1 - \theta}{2}$. In addition, $A(\sigma)$ is non-increasing for $\sigma \geq \sigma_0$, so the function on the left-hand side in \eqref{eq: case2-u<v} is non-increasing. Thus, it is sufficient to check the inequality \eqref{eq: case2-u<v} for $\sigma = \sigma_0$. We will get
$$\frac{1}{2}(1 - \theta - u) \leq (1 - u - v)(1 - \sigma_0),$$
which is equivalent to $(A3)$.

By reversing roles of $u$ and $v$ we prove \eqref{eq: EstLogs-case2} for the case $$v \leq (1 - \theta) \left( \frac{1}{2} + A(\sigma)(1 - \sigma) \right) \leq u.$$ In case $u, v \leq (1 - \theta) \left( \frac{1}{2} + A(\sigma)(1 - \sigma) \right) - u$ we rewrite \eqref{eq: EstLogs-case2} as follows
$$\frac{\frac{1-\theta}{2} - u}{1 - \sigma} + \frac{\frac{1-\theta}{2} - v}{1 - \sigma} + 2 A(\sigma) (1 - \theta) \leq 2 (1 - u - v).$$
Again, the left-hand side of the inequality is non-increasing, so it is enough to substitute $\sigma = \sigma_0$ and get the inequality
$$1 - \theta - \frac{1}{2}(u + v) \leq (1 - u - v)(1 - \sigma_0),$$
which is equivalent to the lower bound on $u + v$ from $(A1)$.
\end{proof}

Now we can easily deduce \cite[Lemma 6]{HBIw79} with $\tau = x^{1 - \theta}$.
\begin{cor}{\cite{HBIw79}} \label{lem: zero-dens-HB-Iw} 
Let $\frac{1}{2} \leq \theta < 1$, and let $M$, $N$, $u$, $v$, $T$, $\{a_m\}$, $\{b_n\}$, $M(s)$, and $N(s)$ be as in Lemma \ref{lem: new-dens-weighted-gen}. Assume the following holds
\begin{align*}
    \frac{14 - 18 \theta}{5} &\leq u + v \leq 1,  &(B1) \\ 
    4 u - 5 v &\leq 9 \theta - 5,  &(B2) \\
    4 v - 5 u &\leq 9 \theta - 5,  &(B3) \\
    \theta &\leq \frac{13}{21}, &(B4)
\end{align*}
Then \eqref{eq: new-dens-weighted-gen} is true for an absolute constant $c$, uniformly for $0 \leq \sigma \leq 1$.
\end{cor}
\begin{proof}
We apply Lemma \ref{lem: new-dens-weighted-gen} with $A(\sigma)$ used in \cite{HBIw79}:
$$A(\sigma) = \begin{cases}
			\frac{2}{2 - \sigma}, & 0 \leq \sigma \leq \frac{3}{4}\\
            \frac{3}{3 \sigma - 1}, & \frac{3}{4} \leq \sigma \leq 1.
		 \end{cases}$$
Then $\sigma_0 = \frac{7}{9}$, so conditions $(B1) - (B4)$ are a particular case of $(A1) - (A4)$ from Lemma \ref{lem: new-dens-weighted-gen}. The function $A(\sigma)$ satisfies conditions 1 and 2 from Lemma \ref{lem: new-dens-weighted-gen} by definition.
\end{proof}

\subsection{On weighted zero-density estimates by Baker and Harman.} \label{subsec: New-Zero-Density: weighted-B-H}
In this section, we generalise and improve upon \cite[Lemma 1(i)]{BakHar96}. Although this lemma looks similar to \cite[Lemma 6]{HBIw79}, it provides a slightly stronger statement than \eqref{eq: new-dens-weighted-gen}:
\begin{equation} \label{eq: BH-zero-dens}
    \sum_{\substack{\rho \\ |\gamma| \leq T}} x^{\beta - 1} |M(\rho) N(\rho)| \ll_A (\log x)^{-A},
\end{equation}
for $A > 0$.
The bound \eqref{eq: BH-zero-dens} clearly implies \eqref{eq: new-dens-weighted-gen} while the converse is not true. Nevertheless, we will show \eqref{eq: BH-zero-dens} is true under slightly modified hypotheses of Lemma \ref{lem: new-dens-weighted-gen}. This will follow from Proposition \ref{prop: zero_dens_transition} and Lemma \ref{lem*: new-dens-weighted} below.

\begin{prop} \label{prop: zero_dens_transition}
Let $M(s)$, $N(s)$ be two Dirichlet polynomials with \eqref{eq: new-dens-weighted-gen} being true for an absolute constant $c$, uniformly for $0 \leq \sigma \leq 1$. Then
$$\sum_{\substack{\rho \\ |\gamma| \leq T}} x^{\beta - 1} |M(\rho) N(\rho)| \ll (\log x)^{c+1}.$$
\end{prop}
\begin{proof}
Let us group together all the non-trivial zeroes of $\zeta$ into sets $\sigma \leq \beta \leq \sigma + \frac{1}{\log x}$ where $\sigma \in \left\{ 0, \frac{1}{\log x}, \frac{2}{\log x}, \dots, \frac{\lceil \log x \rceil - 1}{\log x} \right\} \subset [0,1]$. The number of such sets does not exceed $\log x + 1$. Then:
\begin{align*}
    & \sum_{\substack{\rho \\ |\gamma| \leq T}} x^{\beta - 1} |M(\rho) N(\rho)| \ll \log x \cdot \sum_{\substack{|\gamma| \leq T \\ \sigma \leq \beta \leq \sigma + \frac{1}{\log x}}} x^{\beta - 1} |M(\rho) N(\rho)| \\
    \ll & \log x \cdot x^{\sigma + \frac{1}{\log x} - 1} x^{1 - \sigma} (\log x)^c = e (\log x)^{c+1},
\end{align*}
where we used \eqref{eq: new-dens-weighted-gen} when switching to the second line.
\end{proof}

The result of Proposition \ref{prop: zero_dens_transition} suggests that switching from $x^{\sigma - 1} \sum_{\substack{\beta \geq \sigma}} |M(\rho) N(\rho)|$ to $\sum x^{\beta - 1} |M(\rho) N(\rho)|$ worsens the upper bound only by a factor $\log x$. Hence, for achieving \eqref{eq: BH-zero-dens} with an absolute constant $A$ we should obtain a result of the form 
\begin{equation} \label{eq: new-dens-weighted_stronger-sigma}
    \sum_{\substack{\rho \\ \beta \geq \sigma, |\gamma| \leq T}} \left| M(\rho) N(\rho) \right| \ll_A x^{1-\sigma} (\log x)^{-A-1}.
\end{equation}
To obtain such a bound, we will have to overcome the following obstacle from Lemma \ref{lem: new-dens-weighted-gen}: when we switch from \eqref{EstPowersAndLogs} to \eqref{EstPowers}, we choose $c \geq B + 1$, and since \eqref{eq: new-dens-weighted_stronger-sigma} implies $c = A+1$, we obtain $-B-2 \geq A$. Thus, for example, if $B > 0$ (which is common to have) then $A < -2$, but we will expect to use \eqref{eq: new-dens-weighted_stronger-sigma} for any $A > 0$. To overcome this problem, Baker and Harman restrict the range for $(u,v)$ to $u + v \leq 1 - \eta$ with $\eta > 0$ some fixed parameter, and implement a small positive parameter $\varepsilon$ into the definition of $T$. The new value of $T = x^{1 - \theta - \ep/2}$ allows to efficiently use the truncated version of Perron's formula from Section \ref{sec: appendix} for the short interval $(x-y,x]$, with $y = x^{\theta + \ep}$. Moreover, as we show in the proof of Lemma \ref{lem*: new-dens-weighted} below, adding $\ep$ also allows us to weaken the condition $N(\sigma, T) \ll T^{A(\sigma)(1-\sigma)} \log^B T$ to $N(\sigma, T) \ll T^{A(\sigma)(1 - \sigma) + o(1)}$.

\begin{lem} \label{lem*: new-dens-weighted}
Let all notations and conditions from Lemma \ref{lem: new-dens-weighted-gen} hold, with \eqref{eq: zero-dens-gen} replaced by the weaker assumption
\begin{equation*}
    N(\sigma, T) \ll T^{A(\sigma)(1 - \sigma) + o(1)},
\end{equation*}
meaning that for all $\delta > 0$, $N(\sigma, T) \ll_{\delta} T^{A(\sigma)(1 - \sigma) + \delta}$ as $T \to \infty$. Let $\eta > 0, ~ 0 < \ep < 1 - \theta$, $T = x^{1-\theta-\ep/2}$, and let us assume that $u + v \leq 1 - \eta$. Then:
\begin{equation} \label{eq: new-dens-weighted_stronger}
    \sum_{\substack{\rho \\ |\gamma| \leq T}} x^{\beta - 1} \left| M(\rho) N(\rho) \right| \ll_{A, \ep, \eta} (\log x)^{-A}, \quad \text{for all } A >0.
\end{equation}
\end{lem}

\begin{proof}
We will show that the inequality \eqref{eq: new-dens-weighted_stronger-sigma} holds uniformly for all $0 \leq \sigma \leq 1$ and use Proposition \ref{prop: zero_dens_transition}. We follow the proof of Lemma \ref{lem: new-dens-weighted-gen} up to \eqref{MeanDirPol}. Let $\delta := \delta(\ep) > 0$ be small enough and to be defined later, then we have $N(\sigma, T) \ll_{\ep} T^{A(\sigma)(1-\sigma) + \delta}$, and thus, similarly to the proof of Lemma \ref{lem: new-dens-weighted-gen}, we define $E = \min\left\{ T, T^{\frac{1}{2} + A(\sigma)(1-\sigma) + \delta}\right\}$.

To prove \eqref{eq: new-dens-weighted_stronger-sigma}, we will show that
$$M^{\frac{1}{2} - \sigma} N^{\frac{1}{2} - \sigma} (M + E)^{\frac{1}{2}} (N + E)^{\frac{1}{2}} \log x \ll_{A, \ep, \eta} x^{1 - \sigma} \log^{-A-1} x,$$
equivalent to
$$(1 + E/M)^{\frac{1}{2}} (1 + E/N)^{\frac{1}{2}} \ll_{A, \ep, \eta} \left( \frac{x}{MN} \right)^{1 - \sigma} \log^{-A-2} x,$$
which follows from the inequality on exponents below
\begin{equation} \label{eq: LogsUpdated}
\frac{1}{2}\max\left( 0, \log_x E - u \right) + \frac{1}{2}\max\left( 0, \log_x E - v \right) + (A+2)\frac{\log \log x}{\log x} \leq ( 1 - u - v) (1 - \sigma).
\end{equation}

We briefly outline what happens in each of the two cases from the proof of Lemma \ref{lem: new-dens-weighted-gen}.\\

\noindent
{\bfseries Case 1.} $0 \leq \sigma \leq \sigma_0$\\
Then $E = T = x^{1 - \theta - \ep/2}$, and \eqref{eq: LogsUpdated} is equivalent to
\begin{equation} \label{eq: EstLogs-stronger-case1}
\frac{1}{2}\max\left( 0, 1 - \theta - \frac{\ep}{2} - u \right) + \frac{1}{2}\max\left( 0, 1 - \theta - \frac{\ep}{2} - v \right) + (A+2)\frac{\log \log x}{\log x} \leq ( 1 - u - v) (1 - \sigma). 
\end{equation}

If $u, v \geq 1 - \theta - \frac{\ep}{2}$, then \eqref{eq: EstLogs-stronger-case1} follows from
$$(A+2)\frac{\log \log x}{\log x} \leq \eta (1 - \sigma_0) \leq ( 1 - u - v) (1 - \sigma),$$
where the first inequality holds for $x \gg_{A, \eta} 1$, and the second inequality follows from $\sigma \leq \sigma_0$ and $u+v \leq 1 - \eta$.

If $u \leq 1 - \theta - \frac{\ep}{2} \leq v$, then \eqref{eq: EstLogs-stronger-case1} can be rewritten as
$$\frac{1}{2}(1 - \theta - u) \leq (1 - u - v)(1 - \sigma) + \frac{\varepsilon}{4} - (A+2)\frac{\log \log x}{\log x},$$
which, for $x \gg_{A, \ep, \eta} 1$, follows from
$$\frac{1}{2}(1 - \theta - u) \leq (1 - u - v)(1 - \sigma).$$
We can complete this case by repeating the argument from Lemma \ref{lem: new-dens-weighted-gen}.

The argument for $v \leq 1 - \theta - \frac{\ep}{2} \leq u$ is similar with the roles of $u$ and $v$ reversed. If $u, v \leq 1 - \theta - \frac{\ep}{2}$, then \eqref{eq: EstLogs-stronger-case1} is 
$$1 - \theta - \frac{1}{2}(u + v) \leq (1 - u - v)(1 - \sigma) + \frac{\ep}{2} - (A+2)\frac{\log \log x}{\log x},$$
and again, with $x \gg_{A, \ep, \eta} 1$, we can proceed as in the argument from Lemma \ref{lem: new-dens-weighted-gen} to complete this case.\\

\noindent
{\bfseries Case 2.} $\sigma_0 \leq \sigma \leq 1$\\
Then $E \leq T^{\frac{1}{2} + A(\sigma)(1 - \sigma) + \delta}$ and the inequality \eqref{eq: LogsUpdated} follows from
\begin{align}
&\frac{1}{2}\max\left( 0, \left(1 - \theta - \frac{\ep}{2}\right) \left( \frac{1}{2} + A(\sigma)(1 - \sigma) + \delta\right) - u \right) \nonumber\\
+ &\frac{1}{2}\max\left( 0, \left(1 - \theta - \frac{\ep}{2}\right) \left( \frac{1}{2} + A(\sigma)(1 - \sigma) + \delta\right) - v \right) \nonumber\\
&\quad\quad\quad\leq ( 1 - u - v) (1 - \sigma) - (A+2)\frac{\log \log x}{\log x}. \label{eq: EstLogs-stronger-case2}
\end{align}

The zero-free regions \eqref{eq: Korobov-Vinogradov} imply that there are no zeroes $\rho$ with $|\gamma| \leq T$ and $\beta \geq 1 - \frac{1}{\log^{3/4} T}$ for $T$ large enough\footnote{We could have used the stronger zero-free regions, but it would not affect the argument.}. Hence, the sum on the left-hand side of \eqref{eq: new-dens-weighted_stronger-sigma} is empty for $\sigma > 1 - \frac{1}{\log^{3/4} T}$ assuming $x \gg_{\ep} 1$, and thus, the inequality holds. Let us assume $\sigma \leq 1 - \frac{1}{\log^{3/4} T} \leq 1 - \frac{1}{(1 - \theta)^{3/4} \log^{3/4} x}$.

Then, in case $u, v \geq \left(1 - \theta - \frac{\ep}{2}\right) \left( \frac{1}{2} + A(\sigma)(1 - \sigma) + \delta\right) - u$, the inequality \eqref{eq: EstLogs-stronger-case2} follows from
$$(A+2)\frac{\log \log x}{\log x} \leq \frac{\eta}{(1 - \theta)^{3/4} \log^{3/4}x} \leq ( 1 - u - v) (1 - \sigma),$$
where the first inequality\footnote{This inequality shows us the limits of the upper bound \eqref{eq: new-dens-weighted_stronger} expressed in terms of zero-free regions. Namely, if we bound $\sum x^{\beta-1}|M(\rho)N(\rho)|$ by $(\text{Er}(x))^{-1}$ and use the strongest possible zero-free regions \eqref{eq: Korobov-Vinogradov}, then $\log \text{Er}(x) \ll \left(\frac{\log x}{\log \log x}\right)^{1/3}$, where we omitted the dependence on $A, \ep, \eta, \theta$ for simplicity.} holds for $x \gg_{A, \eta} 1$, and the second follows from the upper bound on $\sigma$ and $u+v \leq 1 - \eta$.

Let $u \leq \left(1 - \theta - \frac{\ep}{2}\right) \left( \frac{1}{2} + A(\sigma)(1 - \sigma) + \delta \right) \leq v$, then \eqref{eq: EstLogs-stronger-case2} can be rearranged as
\begin{align} \label{eq: updated-case2-u<v>}
    \frac{1}{2} (1 - \theta) \left( \frac{1}{2} + A(\sigma)(1 - \sigma) \right) &- \frac{1}{2} u \leq ( 1 - u - v) (1 - \sigma) \nonumber \\
    &+\frac{\varepsilon}{4} \left( \frac{1}{2} + A(\sigma)(1 - \sigma) + \delta \right) - \frac{\delta}{2}(1-\theta) - (A+2) \frac{\log \log x}{\log x},
\end{align}
where the second line is non-negative for $x \gg_{A} 1$ under the assumption
$$\frac{\delta}{2}(1-\theta) < \frac{\varepsilon}{4} \left( \frac{1}{2} + A(\sigma)(1 - \sigma) + \delta \right),$$
which holds for all $\sigma_0 \leq \sigma \leq 1$ if we define $\delta(\ep) = \frac{\ep}{4(1 - \theta)}$, for example. Thus \eqref{eq: updated-case2-u<v>} follows from $$\frac{1}{2} (1 - \theta) \left( \frac{1}{2} + A(\sigma)(1 - \sigma) \right) - \frac{1}{2} u \leq ( 1 - u - v) (1 - \sigma),$$
which can be proved by repeating the argument from Lemma \ref{lem: new-dens-weighted-gen}.
The proof is identical for the case $v \leq \left(1 - \theta - \frac{\ep}{2}\right) \left( \frac{1}{2} + A(\sigma)(1 - \sigma) + \delta \right) \leq u$.

In case $u, v \leq \left(1 - \theta - \frac{\ep}{2}\right) \left( \frac{1}{2} + A(\sigma)(1 - \sigma) + \delta \right) - u$ we rewrite \eqref{eq: EstLogs-stronger-case2} as below
\begin{align*} 
    \frac{1}{2} (1 - \theta) \left( \frac{1}{2} + A(\sigma)(1 - \sigma) \right) &- \frac{1}{2} (u+v) \leq ( 1 - u - v) (1 - \sigma) \\
    &+\frac{\varepsilon}{2} \left( \frac{1}{2} + A(\sigma)(1 - \sigma) + \delta \right) - \delta (1-\theta) - (A+2) \frac{\log \log x}{\log x},
\end{align*}
which again, for $x \gg_{A} 1$ and $\delta(\ep) = \frac{\ep}{4(1 - \theta)}$, follows from the analogous case from Lemma \ref{lem: new-dens-weighted-gen}.
\end{proof}


We now cite and prove Lemma 1(i) from \cite{BakHar96} below.
\begin{cor} {\cite[Lemma 1(i)]{BakHar96}} \label{lem: dens-weighted-BakHar}
Let $\eta > 0$, $\varepsilon > 0$, $T = x^{1 - \theta - \ep/2}$, $M = x^u$, and $N = x^v$ be such that the following conditions hold
\begin{align*}
    &\frac{1}{7}(20 - 26\theta) \leq u + v \leq 1 - \eta, \\ 
    &6u - 7v \leq 13 \theta - 7, \\
    &6v - 7u \leq 13\theta - 7, \\
    &\theta \leq \frac{19}{31}.
\end{align*}
Let $\{a_m\}$, $\{b_n\}$, $M(s)$, and $N(s)$ be as at the beginning of Section \ref{subsec: New-Zero-Density: weighted-HB-Iw}. Then
\begin{equation} \label{eq: new-dens-weighted}
    \sum_{\substack{\rho \\ |\gamma| \leq T}} x^{\beta - 1}\left| M(\rho) N(\rho) \right| \ll_{A, \eta, \varepsilon} (\log x)^{-A}, \quad \text{for all } A > 0.
\end{equation} 
\end{cor}

    
\begin{proof}
The proof follows from Lemma \ref{lem*: new-dens-weighted} with $\sigma_0 = \frac{10}{13}$. Such a value of $\sigma_0$ is obtained by applying the zero-density estimates \cite[Theorem 11.5]{Ivic1985}, which were used in the work \cite{BakHar96} of Baker and Harman.
\end{proof}

Using Lemma \ref{lem*: new-dens-weighted} and the zero-density estimates \cite[Theorem 11.4]{Ivic1985}, we provide a slight improvement of Corollary \ref{lem: dens-weighted-BakHar}.

\begin{cor} \label{cor: new-dens-weighted}
Let $\eta > 0$, $\varepsilon > 0$, $T = x^{1 - \theta - \ep/2}$, $M = x^u$, and $N = x^v$ be such that the following conditions hold
\begin{align*}
    &\frac{1}{9}(26 - 34\theta) \leq u + v \leq 1 - \eta, \\
    &8u - 9v \leq 17 \theta - 9, \\
    &8v - 9u \leq 17\theta - 9,\\
    &\theta \leq \frac{25}{41}.
\end{align*}
Let $\{a_m\}$, $\{b_n\}$, $M(s)$, and $N(s)$ be as in the beginning of Section \ref{subsec: New-Zero-Density: weighted-HB-Iw}. Then \eqref{eq: new-dens-weighted} holds.
\end{cor}
\begin{proof}
We apply Lemma \ref{lem*: new-dens-weighted} with $\sigma_0 = \frac{13}{17}$ corresponding to $A(\sigma)$ from \cite[Theorem 11.4]{Ivic1985}.
\end{proof}

We note that the recent result by Guth and Maynard \cite{GuthMaynard2024}, when applied to Lemma \ref{lem*: new-dens-weighted}, gives a larger value of $\sigma_0$ and thus more narrow regions for $(u,v)$. The reason is that \cite{GuthMaynard2024} is much better on the left of $\frac{3}{4}$, whereas the result in \cite[Theorem 11.4]{Ivic1985} is better to the right of $\frac{3}{4}$, where the intersection of $A(\sigma)$ and $\frac{1}{2 - 2\sigma}$ is taken. It seems interesting to explore how to generalise Lemma \ref{lem*: new-dens-weighted} to make use of the result by Guth and Maynard. To produce such a generalisation, one idea would be to look into variations of Lemma \ref{lem: new-dens-weighted-gen} provided in \cite{HBIw79}. Namely, using the same argument but replacing the Cauchy--Schwarz inequality by the H\"{o}lder inequality with various weights in the bound for $\sum|M(\rho)N(\rho)|$, we can get other regions for $u, v$ where \eqref{eq: new-dens-weighted} is true. In \cite{HBIw79}, the authors use two more choices for weights, so that \eqref{eq: Holder-2} is replaced by either
$$\sum_{|\gamma| \leq T} |M(\rho) N(\rho)| \leq N(\sigma, T)^{1/4}\left( \sum \left| M(\rho) \right|^2 \right)^\frac{1}{2} \left( \sum \left| N(\rho) \right|^4 \right)^\frac{1}{4},$$
or
$$\sum |M(\rho) N(\rho)| \leq N(\sigma, T)^{1/3}\left( \sum \left| M(\rho) \right|^2 \right)^\frac{1}{2} \left( \sum \left| N(\rho) \right|^6 \right)^\frac{1}{6}.$$

In these two cases, the authors of \cite{BakHar96} used the same zero-density results as in \cite{HBIw79} --- even if it was possible to get a slight improvement, the authors \cite{BakHar96} noted it would not lead to a major improvement in the final result. It is also harder to provide generalisations similar to Lemmas and \ref{lem: new-dens-weighted-gen} and \ref{lem*: new-dens-weighted}: we see the function $N(\sigma, T)$ on the right-hand side of both H\"{o}lder's inequalities above. Because of this, we will need to impose more conditions on the function $A(\sigma)$ to derive the good regions for $(u, v)$. In particular, the regions from \cite[Lemma 1(ii)-(iii)]{BakHar96} do not depend on the value of $\sigma_0$ anymore but on the maximum and some more precise monotonicity properties of $A(\sigma)$. 

Nevertheless, considering the recent major improvement of the zero-density estimates \cite{GuthMaynard2024}, a good optimisation question to answer would be:
\begin{question} \label{question: Holder-for-regions}
\textnormal{
    What weights in the H\"{o}lder inequality are the most efficient and can make the use of the best available zero-density estimates? What weights will provide a non-negligible contribution to the regions $(u, v)$ where \eqref{eq: new-dens-weighted} holds?
}
\end{question}

\section{Appendix} \label{sec: appendix}

In this section, we explain how to use the weighted zero-density estimates from Section \ref{sec: New-Zero-Density} to estimate the sums
$$\sum_{\substack{m \sim M, n \sim N \\ x-y < mnr \leq x}} a_m b_n \Lambda(r),$$
where $y = x^{\theta + \ep}$ for some $0 < \theta < 1$, $0 < \ep < 1 - \theta$, and $x \gg_\theta 1$ is large enough.

\begin{prop}
Let all notations and conditions from Lemma \ref{lem*: new-dens-weighted} hold, then 
\begin{equation*}
    \sum_{\substack{m \sim M \\ n \sim N \\ x-y < mnr \leq x}} a_m b_n \Lambda(r) = y \sum_{\substack{m \sim M \\ n \sim N}}\frac{a_m b_n}{mn} + O_{A, \eta, \ep}(y \log^{-A}x).
\end{equation*}
In particular \eqref{eq: PNT-weighted} holds if for some $C > 0$,
$$\sum_{\substack{m \sim M \\ n \sim N}}\frac{a_m b_n}{mn} \gg \log^{-C}x.$$
\end{prop}
\begin{proof}
    Let $M(s) =\sum_{m \sim M} a_m$, $N(s) = \sum_{n \sim N} b_n$, where $0 \leq a_m, b_n \leq 1$. By the truncated version of Perron's formula \cite[Lemma 3.12]{Tit86}, for all $T \geq 1$,
\begin{equation*}
    \sum_{\substack{m \sim M, n \sim N \\ x-y < mnr \leq x}} a_m b_n \Lambda(r) = y \sum_{\substack{m \sim M \\ n \sim N}}\frac{a_m b_n}{mn} - \sum_{\substack{\rho \\ 0 \leq \beta \leq 1 \\ |\gamma| < T}} \frac{(x+y)^{\rho} - x^{\rho}}{\rho} M(\rho)N(\rho) + O\left( \frac{x \log^2 x}{T} \right).
\end{equation*}
By taking $T = (y/x)x^{\ep} = x^{1-\theta - \ep/2}$ and using \eqref{eq: mean-value-x-rho}, we get
\begin{align*}
    \sum_{\substack{m \sim M, n \sim N \\ x-y < mnr \leq x}} a_m b_n \Lambda(r) - y \sum_{\substack{m \sim M, n \sim N \\ x-y < mnr \leq x}}\frac{a_m b_n}{mn} &\leq y \sum_{\substack{\rho \\ 0 \leq \beta \leq 1 \\ |\gamma| < T}} x^{\beta - 1} |M(\rho)N(\rho)| + O\left(y ~\frac{\log^2 x}{x^\ep}\right),\\
    & = O_{A, \eta, \ep}(y \log^{-A}x),
\end{align*}
where we used Lemma \ref{lem*: new-dens-weighted} to switch to the last line.
\end{proof}

\section*{Acknowledgements}

I would like to thank Timothy Trudgian and Bryce Kerr for their support during my work on this paper. I would also like to thank Roger Heath-Brown and Glyn Harman for their help and explanations via the emails and personal chats during my visit to the UK, and James Maynard for hosting me and having mathematical chats with me at the University of Oxford during this visit. Finally, I would like to thank the reviewers of my PhD thesis for their valuable comments, which resulted in the improved version of this paper.

\bibliographystyle{acm}
\bibliography{refs.bib}

\end{document}